%% file: sphericalharmonics.tex
\documentclass[preprint]{elsarticle}

\usepackage{graphicx}
\usepackage{amsmath}
\usepackage{shortcuts}
\usepackage{amssymb}
\usepackage{amsthm}
\usepackage{multirow}

\newtheorem{thm}{Theorem}
\newtheorem{lemma}{Lemma}
\newtheorem{prop}{Proposition}

\newtheorem*{mrthm}{Theorem (Morales-Ramis)}
\newtheorem*{kthm}{Theorem (Kovacic)}

\begin{document}

\begin{frontmatter}
\title{Regular and irregular geodesics on spherical harmonic surfaces}
\author{Thomas J. Waters}
\address{Department of Mathematics, University of Portsmouth,\\ Portsmouth PO1 3HF, United Kingdom}
\begin{abstract}
  The behavior of geodesic curves on even seemingly simple surfaces can be surprisingly complex.
  In this paper we use the Hamiltonian formulation of the geodesic equations to analyze their
  integrability properties. In particular, we examine the behavior of geodesics on surfaces defined
  by the spherical harmonics. Using the Morales-Ramis theorem and Kovacic algorithm we are able to
  prove that the geodesic equations on all surfaces defined by the sectoral harmonics are not
  integrable, and we use Poincar\'{e} sections to demonstrate the breakdown of regular motion.

\end{abstract}
\end{frontmatter}


\section{Introduction}

The study of the qualitative behavior of geodesics goes back to
the beginnings of differential geometry; in fact Clairault's
relation dates from the 1740's. One of the many interesting
features of the geodesic equations is their complexity: how can we
solve such difficult nonlinear differential equations, and indeed
what does it mean to `solve' a differential equation at all? In
the general development of Lagrangian/Hamiltonian dynamics the
issue of integrability came to the fore, and the use of first
integrals to solve a dynamical system via quadratures was
developed. By integrability here we mean intebrability in the
sense of Louiville: a Hamiltonian system with $n$ degrees of
freedom is Louiville integrable if there exist $n$ independent
first integrals of the motion in involution (vanishing pairwise
Poisson brackets). This formalism in turn lead to the consequences
of non-integrability, the onset of chaos and the celebrated KAM
theory. For a particularly interesting introduction to this topic
see Berry \cite{berry}, McCauley \cite{mccauley} or Jos\'{e} and
Saletan \cite{jose}.

The integrability of the geodesic equations on some surfaces, and
therefore the regularity of the geodesic trajectories, has been
well understood for many years. Clairault's relation
\cite{docarmo} provides a second integral for surfaces of
revolution, and Jacobi \cite{kling} integrated the triaxial
ellipsoid. These results have been extended to $n$-dimensional
ellipsoids and to quadrics in general (see Audin \cite{audin}),
where interestingly degenerate cases can prove more difficult to
analyse (see Davison et al \cite{davison}). Recently, the
integrability of the geodesic equations has been demonstrated for
more abstract manifolds, defined in terms of Lie groups (see for
example Thimm \cite{thimm}). Aside from being interesting in their
own right, the geodesics of a manifold also have some useful
applications. For example, Maupertius's principle \cite{bols}
relates the dynamics on an isoenergetic level of any mechanical
system to the geodesics of a manifold, and the relativistic
formulation of gravity dictates that massive particles move along
the timelike geodesics of spacetime \cite{wald}.

There are many powerful results regarding the non-integrability of
the geodesic flow\footnote{Technically, the geodesic flow is
defined on the cotangent bundle whose projection onto the manifold
gives the geodesic curves, however we will use these two tems
interchangeably.} on manifolds of a certain class, for example a
theorem of Kolokol'tsov \cite{kolok} says that two-dimensional
compact surfaces of genus $g>1$ do not admit geodesic flows that
are integrable with a second integral polynomial with respect to
the momenta (see Bolsinov \cite{bols2} for a review and
references). What's more the integrable geodesic flows on surfaces
with $g=0,1$ with second integrals either linear or quadratic in
the momenta are completely described. Nonetheless, new examples of
integrable geodesic flows with second integrals cubic in the
momenta have recently been found (see for example Dullin and
Matveev \cite{dullin}) and the possibility of non-polynomial, and
in particular meromorphic second integrals, is open.

A very significant recent development with regards integrability
has been the theorem of Morales Ruiz and Ramis
\cite{mrthm1,mrthm2}. This theorem provides a link between the
question of whether a nonlinear Hamiltonian system is {\it
integrable} (with meromorphic first integrals) with whether the
variational equations of a non-trivial solution to the Hamiltonian
system are {\it solvable}. This is a great advance as the
variational equations are linear with variable coefficients, and
the solvability of equations of this type may be addressed using
differential Galois theory. This approach has proved very
successful (we give here only some examples): in Morales Ruiz and
Ramis \cite{mrthm2} the authors use this formalism to prove the
non-integrability of certain problems from mechanics and celestial
mechanics; in Morales Ruiz et al \cite{simon} the non
integrability of Hill's problem is shown; in Pujol et al
\cite{sam} the non-integrability of the swinging Atwood's machine
is demonstrated; in Acosta-Hum\'{a}nez et al \cite{humanez} some
problems in celestial mechanics are studied and in Bardin et al
\cite{bardin} the generalized Jacobi problem is studied. In the
majority of these papers, the solvability of the variational
equations is determined using Kovacic's algorithm, and this is the
approach we will follow in this work, with a description of the
algorithm and relevant references given in Section 4.

To the best of our knowledge this powerful Morales-Ramis formalism
has not been applied to the study of geodesics on manifolds and
their integrability properties, and the approach taken in this
work is novel in this respect. In the following we will examine
the properties of geodesics on surfaces defined in terms of the
spherical harmonics, and we shall rigorously prove using the
Morales-Ramis theorem the non-integrability of the geodesic
equations on surfaces defined in terms of sectoral spherical
harmonics. Thus we present a two-parameter family of surfaces
which appear very simple and symmetrical and yet have complex and
indeed chaotic geodesic curves on them; this is a novel addition
as the surfaces for which non-integrability of the geodesic flow
has been rigorously proved is small. Finally, we will also use
Poincar\'{e} sections to display the breakdown of regular motion
as well as to understand the various families of closed geodesics
on the surfaces; again this approach is novel in the case of
geodesic flows.

The layout of the paper shall be as follows: In Section 2 we
review the formulation of the geodesic equations and describe the
classes of surfaces which will be examined. In Section 3 we use
Poincar\'{e} sections to exhibit the breakdown in regular motion
suggesting non-integrability, as well as describe the main
families of closed geodesics on the surface. Finally in Section 4
we will review some necessary theory and prove the main result:
the non-integrability of the geodesic equations on sectoral
harmonic surfaces with $n>1$.


\input{psecsec}


\input{kovsec}


\section{Conclusions}

The application of Morales-Ramis theory and Poincar\'{e} sections,
two techniques most often seen in the realm of celestial mechanics
and mechanical systems, to study the integrability of geodesic
flow presents many interesting possibilities. By choosing a
specific class of surfaces to analyze we can rigorously prove the
non-integrability of the geodesic equations. It would be of
interest to extend the results of this work to examine further
surfaces; for example the tesseral harmonics mentioned in Section
2 again have (meridianal) planes of symmetry, and the NVE will
have coefficients that contain Legendre polynomials, and in
general, it would be very informative if we could examine
arbitrary surfaces with a plane of symmetry and try and arrive at
some obstructions to integrability. We note that the presence of a
plane of symmetry makes the decoupling of the NVE straightforward,
however Morales Ruiz \cite{morruiz} outlines a procedure in the
absence of an invariant plane. The completeness of the spherical
harmonics means that surfaces of a very broad class can be
decomposed into a set of spherical harmonics, and perhaps by
analyzing these individually we can make some conclusions about
their sum. Care must be taken however: the triaxial ellipsoid can
be decomposed into a set of (infinite) spherical harmonics, each
of which would likely have non-integrable geodesic flow, and yet
their recombination gives a surface which {\it is} integrable. A
further issue to consider is how the surface in question is
described; the parametric description of the harmonic surfaces led
easily to the calculation of the geodesic equations and in
particular the separation of the normal variational equation. If
the surface were defined in algebraic form, as is more suited to
the double torus for example, this separation may be more
problematic. Finally, this subject provides an interesting and
fruitful combination of techniques from dynamical systems theory,
differential geometry and differential Galois theory.

\section{Acknowledgements}

The author wishes to thank the School of Mathematics, National
University of Ireland, Galway, where some of this work was carried
out; Dr.\ Sergi Simon for useful conversations; and finally the
referees for their important and constructive comments.


\appendix

\section{Meridianal sections}

The analysis above would carry through were we to use one of the
meridianal planes instead of the equatorial plane, with two points
to note:\begin{enumerate}
\item When drawing the Poincar\'{e} sections, the fact that
$\phi=$constant is only a half plane complicates matters. Better
to rotate the coordinate system (or surface) to make the meridian
an equator. For this we need to rotate the surface about the
$x$-axis, and for this we define a new set of coordinates
$\vartheta,\varphi$ by \[ \sin\theta=\sqrt{1-\sin^2\vartheta
\sin^2\varphi},\quad \sin\phi=\frac{\cos\vartheta}{\sin\theta},
\quad \cos\phi=\frac{\sin\vartheta \cos\varphi}{\sin\theta}.
\] Now to write the $n$th sectoral harmonic we expand
$\cos(n\phi)$ and then make the replacement; for example, the 3rd
sectoral harmonic changes as
\begin{align*} \hspace{-0.5cm} \sin^3\theta \cos(3\phi)=\sin^3\theta
(\cos^3\phi-3\cos\phi\sin^2\phi)=\sin^3\vartheta\cos^3\varphi-3\sin\vartheta\cos\varphi\cos^2\vartheta.
\end{align*} An example of a Poincar\'{e} section for this rotated
surface is shown below; the same families of closed geodesics can
be identified.

\begin{figure}[h]
\begin{center}
\includegraphics[width=0.8\textwidth]{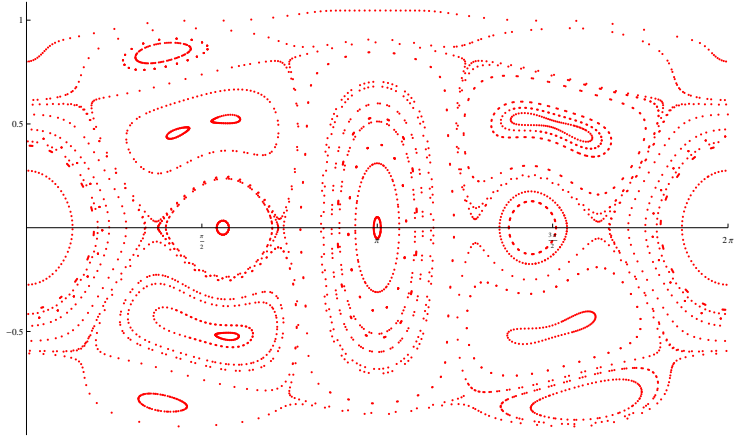}
\caption{Sample Poincar\'{e} section for the rotated $n=3$
sectoral harmonic surface with $\vep=0.1$. The elliptic nature of
the equatorial geodesic (the fixed point at the centre of the
plot) can be clearly seen.}
\end{center}
\end{figure}

\item In deriving the variational
equations we may linearize around a meridianal planar geodesic, to
derive a normal variational equation of the form given in
\eqref{nve2}. In the case of an equatorial planar geodesic the
denominators of the coefficients are at most quadratic in $z$,
however in the case of a meridianal planar geodesic the
denominators are polynomials of degree $2n+2$. While the singular
points are still regular, the number, location and behaviour of
the set of singular points will vary with $n$, complicating
matters.
\end{enumerate}

\section{Coefficients of equation \eqref{pfe}}

\begin{align*}
  \delta_1&=\frac{2}{n(\vep^2-1)}, \\
  \delta_{2/3}&=\pm\frac{8+n^2\pm 2\vep(8+4n+3n^2)+e^2(8+8n+5n^2+8n^3+4n^4)}{16\vep(1\pm\vep)^2 n^2}, \\
  \delta_{4/5}&=-\frac{1}{2n(1-\vep^2)^2}\left( \vep^2(1-2n^2)-1\pm\big[n(1+\vep^2)\sqrt{1+\vep^2(n^2-1)}\,\big]
  \right) \\ &\qquad -\frac{1}{8}\left(\frac{4}{a_{4/5}-a_1}+\frac{1}{a_{4/5}-a_2}+\frac{1}{a_{4/5}-a_3}-\frac{1}{a_{4/5}-a_{5/4}}
  \right),
\end{align*} where the $a_j$ are as defined in
\eqref{app1},\eqref{app2}.

\section{Bibliography}

\bibliographystyle{plain}
\bibliography{geobib}

\end{document}

%% file: psecsec.tex
\section{The geodesic equations}

One geometrical description of a geodesic curve is as follows: of
all the curves connecting two nearby points on a manifold, the one
that extremizes the length along that curve is the geodesic
(minimizes in the case of a positive definite manifold). An
alternative dynamical description is to imagine a test particle
moving in the manifold free from external forces; if the
particle's speed is constant, then the path it follows is a
geodesic. Consequently, we may define a (dynamic) Lagrangian
$\mathcal{L}$ in the following way: let the coordinates $x^a$
parameterize the $n$-dimensional manifold $(\mathcal{M},\bmt{g})$
and let $s$ parameterize a curve in the manifold, such that
$x^a=x^a(s)$. The length of the tangent vector to this curve is
\bes \left(\sum_{a,b=1}^{n} g_{ab}\ \dot{x}^a\,
\dot{x}^b\right)^{1/2}, \ees but this is simply the speed. Here
$g_{ab}$ are the components of the metric tensor in the $x^a$
coordinate system, and from now on a dot denotes differentiation
w.r.t.\ $s$ and summation over repeated indices is implied. We may
define the kinetic energy as one half the square of this quantity
(assuming unit mass), and therefore the Lagrangian function is \be
\mathcal{L}=\tfrac{1}{2} \ g_{ab}\ \dot{x}^a\, \dot{x}^b. \ee
Putting the Lagrangian function into the Euler-Lagrange equations
we arrive at the geodesic equations, \be
\der{^2x^a}{s^2}+\Gamma^a_{bc}\ \der{x^b}{s}\der{x^c}{s}=0,
\label{geoeqn} \ee where the Christoffel symbols (metric
connections) are given by \be
\Gamma^a_{bc}=\tfrac{1}{2}g^{ad}(g_{bd,c}+g_{cd,b}-g_{bc,d}).
\label{chris}\ee The equivalent Hamiltonian formulation is clear,
by simply letting $\mathcal{H}=\mathcal{L}$. For a more detailed
derivation of the geodesic equations see Paternain
\cite{paternain}, Do Carmo \cite{docarmo} or alternatively Wald
\cite{wald}. We note that some authors prefer to use the
coordinates of the space the manifold is embedded in, and place a
constraint on these coordinates (see for example Davison et al
\cite{davison}). This approach does not suit our purposes here.

\subsection{Surfaces defined in polar form}

Let us consider two-dimensional surfaces given in polar form, that
is \be r=r(\theta,\phi), \ee where $(\theta,\phi)$ are the
standard spherical coordinates with $0<\theta<\pi,\ 0<\phi< 2\pi$.
We may calculate the geodesic equations in the following way: the
line element for $\mathbb{R}^3$ in spherical coordinates is given
by \bes ds^2=dr^2+r^2d\theta^2+r^2\sin^2\theta d\phi^2. \ees On
the surface we have $r=r(\theta,\phi)$ and hence
$dr=r_{,\theta}d\theta+r_{,\phi}d\phi$ and so \be
ds^2=(r_{,\theta}^2+r^2)d\theta^2+2r_{,\theta}r_{,\phi}d\theta
d\phi+(r_{,\phi}^2+r^2\sin^2\theta)d\phi^2, \label{legenr} \ee
from which the metric coefficients may be read off, inserted into
the Christoffel symbols in \eqref{chris} which in turn are put
into the geodesic equations in \eqref{geoeqn}, namely
\begin{align}
\ddot{\theta}+\Gamma^{\theta}_{\theta\theta}\dot{\theta}^2+2\Gamma^{\theta}_{\theta\phi}\dot{\theta}\dot{\phi}
      +\Gamma^{\theta}_{\phi\phi}\dot{\phi}^2=0, \label{geoth} \\ \ddot{\phi}+\Gamma^{\phi}_{\theta\theta}\dot{\theta}^2+2\Gamma^{\phi}_{\theta\phi}\dot{\theta}\dot{\phi}
      +\Gamma^{\phi}_{\phi\phi}\dot{\phi}^2=0. \label{geoph} \end{align}

It will be most informative for later analysis to describe closed
geodesics on these surfaces. For this, we note the following: if a
smooth surface has a plane of symmetry, then the curve formed by
the intersection of the plane of symmetry with the surface will be
a geodesic. Geometrically the reason for this is clear: if the
surface is smooth the plane of symmetry will be normal to the
surface, and any normal curve is geodesic. For the sake of the
analysis to follow however we will make the following argument:
let us arrange the coordinate system so that the plane of symmetry
is given by $\theta=\pi/2$; on this plane $r_{,\theta}=0$. From
\eqref{geoth}, we see $\theta=\pi/2,\dot{\theta}=0$ will be a
solution if $\Gamma^\theta_{\phi\phi}(\theta=\pi/2)=0$. Direct
calculation from \eqref{chris} and \eqref{legenr} gives \[
\Gamma^\theta_{\phi\phi}=\frac{r_{,\theta}(r r_{,\phi\phi}
\sin^2\theta -2r_{,\phi}^2\sin^2\theta-r^2\sin^4\theta
)-\cos\theta(r^3\sin^3\theta+r r_{,\phi}^2\sin\theta
)}{r(r^2\sin^2\theta+r_{,\theta}^2\sin^2\theta+r_{,\phi}^2)},
\] which vanishes on $\theta=\pi/2$. This is equivalent to the `invariant plane' construction typical in dynamics which
we will describe more of below;
  we will subsequently refer to such geodesics as {\it planar} geodesics.\\
      Of course, to describe the geodesic completely we must solve the geodesic equations reduced to this plane, namely
      $\ddot{\phi}+\Gamma^{\phi}_{\phi\phi}(\theta=\pi/2)\dot{\phi}^2=0$;
      we will address this issue in Section 4.

\subsection{Surfaces defined by the spherical harmonics}

The surfaces we will examine in this work are those defined by the
spherical harmonics in the following polar form \be
r(\theta,\phi)=1+\vep\, Y_{lm}(\theta,\phi). \ee Here $\vep$ is a
small parameter and $Y_{lm}$ is the spherical harmonic of degree
$l$ and order $m$ (see Wang and Guo \cite{wang}) given by \[
Y_{lm}(\theta,\phi)=\sqrt{\frac{(2l+1)(l-m)!}{4\pi(l+m)!}}\
P^m_l(\cos\theta) e^{im\phi},\quad m=0,\pm1,\ldots,\pm l, \] where
$P^m_l$ are the associated Legendre polynomials. We may absorb the
orthonormality coefficient into $\vep$ and take the real part of
the exponential to write \be r(\theta,\phi)=1+\vep\,
P^m_l(\cos\theta) \cos(m\phi), \quad 0\leq \vep<1. \label{genpol}
\ee With a slight abuse of language, we shall refer to surfaces
defined in terms of the spherical harmonic functions as `spherical
harmonic surfaces' and so on. We chose to study these surfaces for
the following reasons:
\begin{enumerate}
  \item The spherical harmonics are complete, that is any function
  continuous over the sphere may be decomposed into a (possibly
  infinite) set of spherical harmonics. As such, general
  surfaces of a certain class can be decomposed into a number of
  spherical harmonic surfaces.
  \item The spherical harmonics naturally have the smoothness and
  regularity required for this analysis, and are non-singular as long as
  $\vep<1$.
  \item The small parameter allows for a perturbative approach: with
  $\vep=0$ the surface is a sphere, and as we let $\vep$ grow we
  are deforming the surface away from the sphere.
  \item The spherical harmonic surfaces are compact, enabling the
  Poincar\'{e} section analysis of the next section.
\end{enumerate}

There are three classes of spherical harmonic: the zonal, the
sectoral and the tesseral (see Figure \ref{harmex} for
representative examples).
\begin{description}
  \item[Zonal] These are the spherical harmonics with $m=0$, and thus $r=1+\vep\, P_l^0(\cos\theta)$. Since $r_{,\phi}=0$,
  this means $\partial \mathcal{L}/\partial \phi=0$, which from the Euler- Lagrange equations leads to
  $\partial\mathcal{L}/\partial\dot{\phi}=$constant. This is of course Clairault's integral, as the zonal harmonic surfaces are surfaces
  of revolution; the motion of the geodesics on zonal harmonic surfaces will therefore be regular.
  \item[Sectoral] These are the spherical harmonics for which $l=m$. We shall relabel the degree/order $n$ in this case, and $Y^n_n$ takes
   on a simple form: the sectoral harmonic surfaces are defined by \be r=1+\vep\, \sin^n(\theta)\cos(n\phi). \label{secpol} \ee This describes a two-parameter family of surfaces which admit planes of symmetry; as such they provide an interesting case for analysis and will form the main part of this work.
  \item[Tesseral] These are the remaining members of the family, with $m=\pm 1,\ldots$,$\pm l-1$. We will not examine these surfaces in any
  detail in this work, but discuss them briefly in the
  conclusions.
\end{description}

\begin{center}
\begin{figure}
\includegraphics[width=0.95\textwidth]{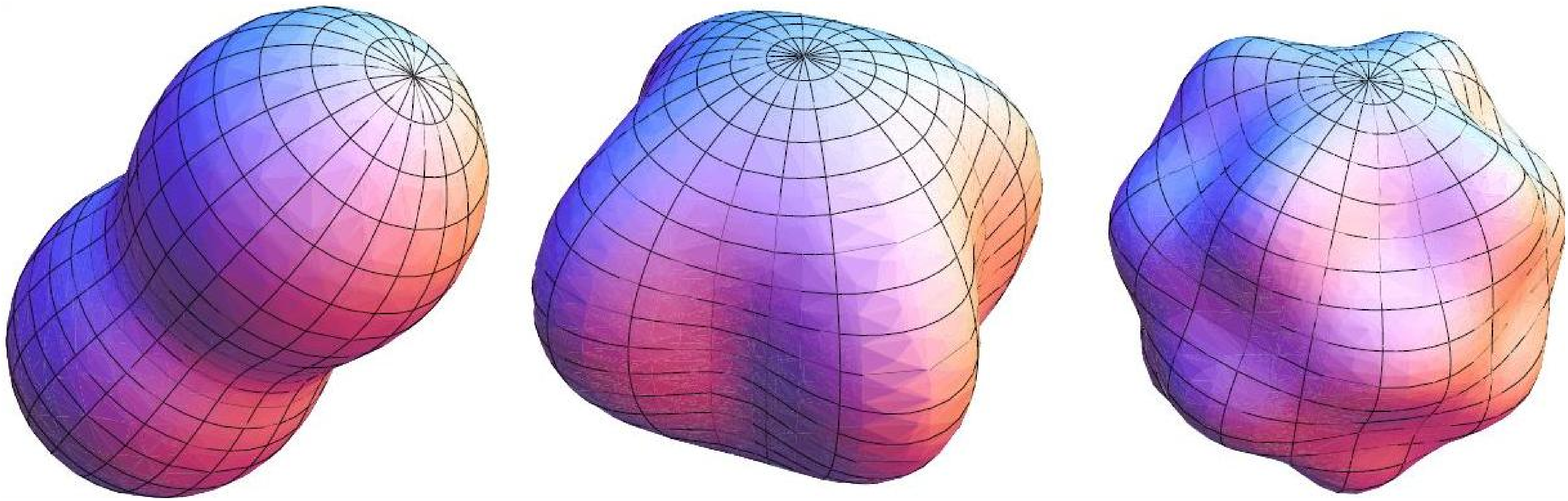}
\caption{Examples of
spherical harmonic surfaces, for various values of $\vep$: on the
left the $l=2,m=0$ zonal harmonic, in the centre the $l=m=4$
sectoral harmonic, and on the right the $l=7,m=4$ tesseral
harmonic.} \label{harmex}
\end{figure}
\end{center}

\vspace{-0.7cm} The Hamiltonian function for the geodesic
equations on the sectoral harmonic surfaces is given by \be
2\ham=g_{\theta\theta}
\dot{\theta}^2+2g_{\theta\phi}\dot{\theta}\dot{\phi}+g_{\phi\phi}\dot{\phi}^2
\ee where \begin{align*}
  &g_{\theta\theta}=(1+\vep\cos(n\phi)\sin^n(\theta))^2+\vep^2n^2\cos^2(n\phi)\sin^{2n-2}(\theta), \\
  &g_{\theta\phi}=g_{\phi\theta}=-\vep^2n^2\cos(n\phi)\sin(n\phi)\cos(\theta)\sin^{2n-1}(\theta), \\
  &g_{\phi\phi}=\sin^2(\theta)(1+\vep\cos(n\phi)\sin^n(\theta))^2+\vep^2n^2\sin^2(n\phi)\sin^{2n}(\theta).
\end{align*}

It is clear from \eqref{secpol}, or indeed Figure \ref{harmex},
that the $n$th sectoral harmonic surface has $n+1$ planes of
symmetry: $n$ vertical/meridianal planes given by $\phi= \pi i/n,\
i=0,\ldots,n-1$ and 1 horizontal plane given by $\theta=\pi/2$,
the equator. Each of these planes defines a planar geodesic as
described above.



\section{Poincar\'{e} Sections}


This is a much used technique to attempt to visualize the whole of
phase space in the case of a 2-degree of freedom Hamiltonian
system. By fixing a value of $\mathcal{H}=\mathcal{H}_0$, a
3-dimensional submanifold of phasespace, and intersecting this
manifold with a hyperplane (typically given by a coordinate taking
on a fixed value), we consider successive intersections of
solution trajectories with this hyperplane. If the Hamiltonian
system is integrable, then the second integral restricts the
solutions further to a 2-dimensional submanifold of phasespace,
whose intersections with the hyperplane will trace out a closed
curve. On the other hand, if the Hamiltonian system is not
integrable then successive intersections of the solution
trajectory with the hyperplane may instead fill a 2-dimensional
region.

Typically, for a fixed set of parameters a number of Poincar\'{e}
sections for various values of the `energy' $\mathcal{H}_0$ must
be drawn, for example in the H\'{e}non-Heiles system \cite{hh}.
This is not necessary in the current situation, since the
numerical value of the Hamiltonian function may be changed from
$\ham_0$ to $\ham_1$ by the change of independent variable $s\to
\sqrt{\ham_0/\ham_1}\, s$. Therefore the numerical value of $\ham$
signifies the {\it parameterization} of the geodesic curve, rather
than any `energy' concept. Indeed, any affine coordinate
transformation $s\to \alpha s+\beta$ will change the value of
$\ham$ but leaves \eqref{geoeqn} invariant. Therefore, any
qualitative change in the Poincar\'{e} sections will be due to
variations in the surface's parameters ($n$ and $\vep$ for the
sectoral harmonic surfaces), rather than the value of $\ham$. In
what follows, we will assume the geodesic is parameterized by
arc-length (sometimes referred to as `unit-speed curves' in the
literature), and therefore \[ \ham\equiv \tfrac{1}{2}. \]

In choosing the hyperplane to act as the Poincar\'{e} section we
must ensure all solutions intersect the hyperplane transversally;
no tangential grazings are allowed, and a guarantee that the
trajectory will intersect the surface again is preferable. Any of
the $n+1$ planes of symmetry would be suitable candidates,
although for simplicity we will use the equatorial plane (some
comments regarding the use of meridianal planes are given in
Appendix A). First we prove the following lemma:

\begin{lemma}
Along a geodesic, $\theta$ cannot obtain a maximum/minimum in the
northern/southern hemisphere respectively, at least for moderate
values of $\vep$.
\end{lemma}
\begin{proof}
We remind the reader that $\theta$ is measured {\it down} from the
positive $z$-axis, and the northern and southern hemispheres are
given by $\theta\in(0,\pi/2)$ or $(\pi/2,\pi)$ respectively. Due
to symmetry, we need only show that $\theta$ cannot obtain a
maximum for $\theta\in(0,\pi/2)$.\\
From \eqref{geoth}, a maximum w.r.t.\ $\theta$ in the northern
hemisphere would require $\Gamma^\theta_{\phi\phi}>0$ for some
$\theta\in(0,\pi/2)$ (on the sphere,
$\Gamma^\theta_{\phi\phi}=-\cos\theta\sin\theta< 0$ for
$0<\theta<\pi/2$, ruling out this possibility). We can write \[
\Gamma^\theta_{\phi\phi}=-\frac{r\cos\theta\sin^3\theta}{det(g)}\
f(\sin\theta,\cos(n\phi);n,\vep). \] As the surface is positive
definite to prove the lemma we must show $f\geq 0$ in
$\Omega=\{0<\theta<\pi/2$, $0<\phi<2\pi\}$. Letting
$\rho=\sin\theta$ with $0<\rho<1$ in $\Omega$, we can write \[
f(\rho)=1+a_1\rho^n+a_2\rho^{2n-2}+a_3\rho^{2n}+a_4\rho^{3n-2}+a_5\rho^{3n}
\] where the $a_i$ depend on $\cos(n\phi),\ n$ and $\vep$ ($\phi$ here
being the coordinate at which the proposed maximum takes place).
When $\cos(n\phi)>0$, all the $a_i$ are positive and so
$f(\rho)>0$ in $\Omega$.\\ If, on the other hand, $\phi$ is such
that $\cos(n\phi)<0$, Descarte's rule of signs tells us that
$f(\rho)$ can only have 1 or 3 roots in $0<\rho<1$, and since
$f(0)>0$ this is impossible unless $f(1)<0$. Letting $f(1)=0$
gives a cubic in $\cos(n\phi)$, and we can then find explicitly
the value of $\vep$ as a function of $n$ at which the value of
$f(1)$ becomes negative for some values of $\phi$. If we call this
critical value $\vep^*$, then for example when $n=2,3,4$ we find
$\vep^*=0.57,0.497,0.445$.
\end{proof}

Note this lemma rules out the possibility of closed geodesics
which do not intersect the equator at some stage (the only
$\theta=$ constant geodesic is the equator). What it does not rule
out, however, is the possibility of a geodesic asymptotically
approaching the equator. In fact, this cannot occur as the equator
is of elliptic type. A rigorous proof of the non-hyperbolic nature
of the equator is too much of a digression; certainly numerical
evidence can be easily provided using Floquet theory (see
\cite{watquasi} for a detailed description) by calculating the
eigenvalues of the monodromy matrix and showing they remain on the
unit circle in the complex plane for moderate values of $\vep>0$
and various values of $n$. Having said that the most immediate
demonstration of the elliptic nature of the equator is to be found
in the Poincar\'{e} section shown in Fig.\ A.3.

\begin{prop} With the equatorial geodesic of elliptic type, any geodesic which intersects the equatorial
plane transversally must re-intersect the equatorial plane an
infinite number of times, at least for moderate values of $\vep$.
\end{prop}
\begin{proof} First we can see that any trajectory which intersects the
equatorial plane transversally cannot subsequently graze this
plane tangentially as it is an invariant plane. If we consider a
geodesic which crosses the equator with $\dot{\theta}<0$ (i.e.\
entering the northern hemisphere) then there are two
possibilities: 1) the geodesic does not pass through the north
pole, and 2) the geodesic passes through the north pole. In the
event of case 1), the geodesic must attain a minimum w.r.t.\
$\theta$ ($\theta$ is an initially positive decreasing continuous
function of arc-length and so must either: attain a minimum; pass
through zero (case 2); or asymptote to some constant value which
is impossible as the surface is compact and there are no
equilibrium points or closed geodesics contained in a hemisphere).
As the geodesic cannot attain a maximum w.r.t.\ $\theta$ in the
northern hemisphere or asymptote to the equator, it must
subsequently re-intersect the equator. In the event of case 2),
consider a geodesic which starts at the north pole. While there is
a degeneracy in defining $\phi$ at the poles there is no problem
in defining $\dot{\phi}$. Following the geodesic in this direction
it must intersect the equator for the reasons given in case 1),
but the same can be said if we follow the geodesic ``backwards'',
and so it must intersect the equator in both directions.
\end{proof}

A final consideration is the ranges of velocities the geodesics
may explore. On any of the planes of symmetry, either
$r_{,\theta}$ or $r_{,\phi}$ vanishes, thus (from \eqref{legenr})
$g_{\theta\phi}=0$. Therefore on the section we have \be
g_{\theta\theta} \dot{\theta}^2+g_{\phi\phi} \dot{\phi}^2=1, \ee
which is an ellipse of semi-axes
$1/\sqrt{g_{\theta\theta\phantom{\phi}}}$ and
$1/\sqrt{g_{\phi\phi}}$. Since $g_{\phi\phi}(\theta=\pi/2)=(1+\vep
\cos(n\phi))^2+\vep^2 n^2 \sin^2(n\phi)$, when choosing our
initial conditions to generate the Poincar\'{e} section plots we
should therefore choose random values for $\dot{\phi}$ in the
range $\dot{\phi}\in(-\dot{\phi}_m,\dot{\phi}_m)$ with
\begin{equation}
\dot{\phi}_m=\left\{ \begin{array}{cc} \frac{1}{1+\vep} & \quad
\textrm{if}\ \vep<\frac{1}{n^2-1} \\ \\
\frac{1}{\sqrt{\frac{n^2(1+\vep^2(n^2-1))}{(n^2-1)}}} & \quad
\textrm{if}\ \vep>\frac{1}{n^2-1}.
\end{array} \right.
\end{equation}

The initial conditions can now be chosen in the following way:
random values for $\phi$ are chosen in the range $(0,2\pi)$,
random values for $\dot{\phi}$ in the range
$(-\dot{\phi}_m,\dot{\phi}_m)$, and the value of $\dot{\theta}$ is
found from $\mathcal{H}=\tfrac{1}{2}$ (we will always take the
positive value to give a `downward' pointing geodesic). The
geodesic with these initial conditions is followed for a long time
and intersections with the equatorial plane with $\dot{\theta}>0$
are monitored. We show in Figure \ref{psecplots} some of the
resulting Poincar\'{e} sections for the case of the $n=3$ sectoral
harmonic surface. Some interesting points to note are:

\begin{enumerate}
\item There are three main families of {\bf closed geodesics};
these appear as fixed points of the Poincar\'{e} section, and are
as follows: Firstly there are $n$ {\it planar} geodesics as
described in Section 2 leading to fixed points at $\phi=\pi
i/n,i=1,\ldots,n$ and $\dot{\phi}=0$. Secondly on the sectoral
surfaces with $n$ odd there are $n$ {\it perpendicular} geodesics,
which are not planar but intersect the equator at right angles,
leading to fixed points roughly midway between those due to the
planar geodesics. Thirdly there are {\it oblique} geodesics which
are simple (period-1) and give
rise to fixed points off the horizontal axis of the Poincar\'{e} section.\\
Of course, there are an infinite number of closed geodesics on
topological spheres (see Bangert \cite{bangert} and Franks
\cite{franks}) and we make no attempt to classify them all.
\item The {\bf stability} of these closed geodesics is perhaps not what would be expected (Note: we are using only the Poincar\'{e}
  sections to draw conclusions about stability; a more rigorous stability analysis would make an interesting topic for
  future research). Consider first the sectoral surfaces with $n$ even, as shown in Figure 1 (b). The planar geodesics alternate
  between those for which $\partial^2 r/\partial \phi^2\leq 0$ and those for which $\partial^2 r/\partial \phi^2\geq 0$ everywhere
  along the curve. In analogy with the surfaces of revolution \cite{docarmo}, we might therefore expect that these geodesics would alternate
  between stability and instability, respectively. However this is not the case: all planar geodesics are in fact stable, leading to
  elliptic fixed points on the Poincar\'{e} section.\\
      If we now consider the sectoral surfaces for $n$ odd, then $\partial^2 r/\partial \phi^2$
      is both positive and negative along each planar geodesic. In fact, these planar geodesics are all unstable; what's more each is
      homoclinic to itself (at least for small $\vep$, see below) and thus admits two elliptic points within its homoclinic loops, see Figure \ref{psecplots}.
      These elliptic points form a fourth family of closed geodesics, and some numerical experimentation reveals these are also period-1.\\
      The perpendicular geodesics mentioned in 1.\ are all stable, leading to elliptic fixed points in the Poincar\'{e} section,
      and most importantly the oblique geodesics are unstable, leading to hyperbolic fixed points. These hyperbolic fixed points
      are heteroclinic to one another as Figure \ref{psecplots} shows (in fact there are complex heteroclinic networks, where up to $2n$ saddles
      are connected to one another), and this is the crucial ingredient leading to chaotic motion, which we shall describe next.
  \item As is clear from Figure \ref{psecplots}, {\bf chaotic} motion takes place for large $\vep$, that is as we deform the surface away from the sphere.
  The onset of chaos
  as depicted in Figure \ref{psecplots} is straight out of a textbook: the invariant manifolds of unstable fixed points intersect leading to a
  heteroclinic
  tangle, and as $\vep$ increases further rational invariant tori are destroyed. This phenomenon is well described in the literature, see for
  example Jos\'{e} and Saletan \cite{jose}, Calkin \cite{calkin} or Berry \cite{berry}. The value of $\vep$ at which this onset of chaotic motion takes
  place varies for different values of $n$, but note that we see chaotic motion for values of $\vep$ less than those mentioned in Lemma 1.
\end{enumerate}

\begin{center}
\begin{figure}
{(a) \includegraphics[width=0.8\textwidth]{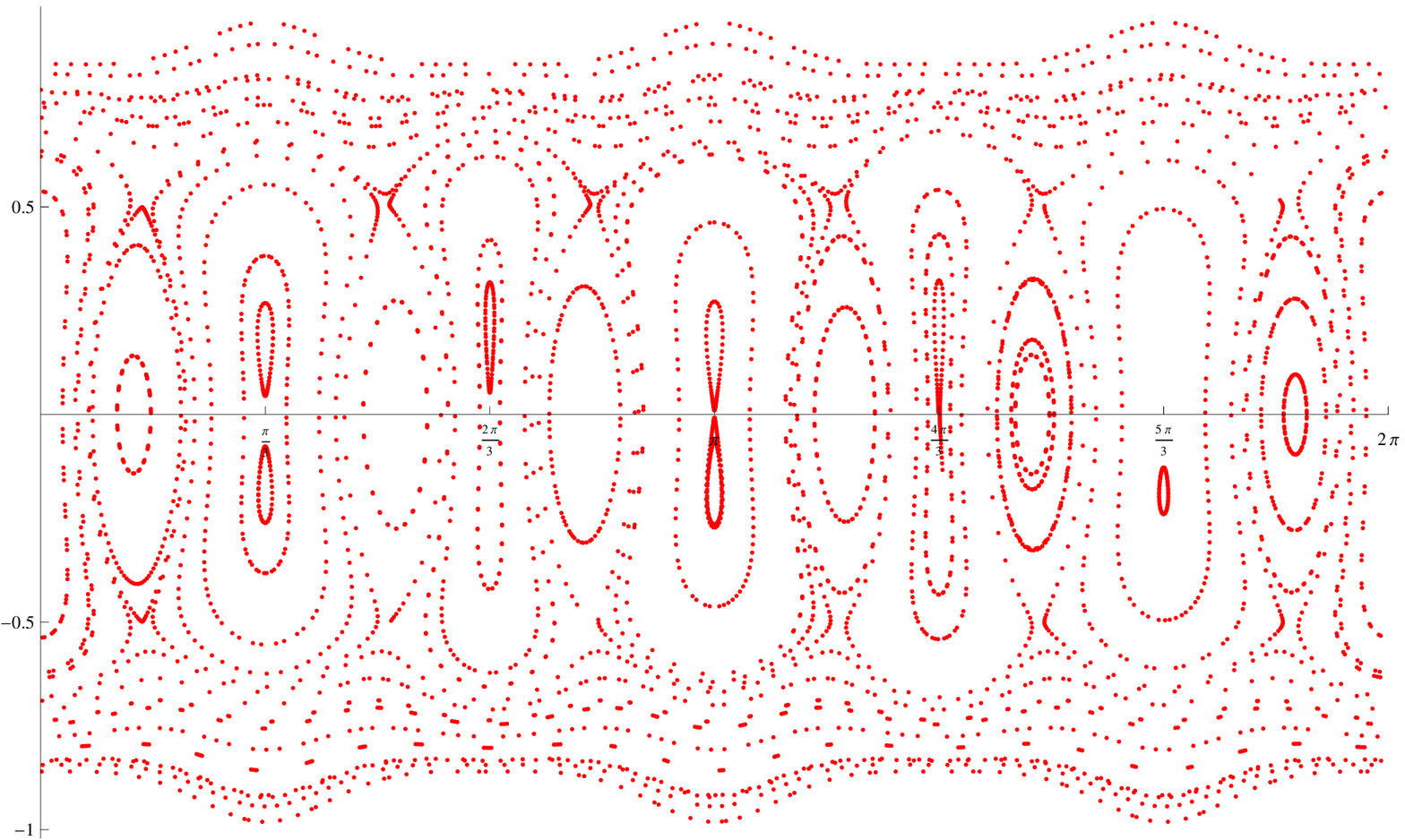} \\
(b) \includegraphics[width=0.8\textwidth]{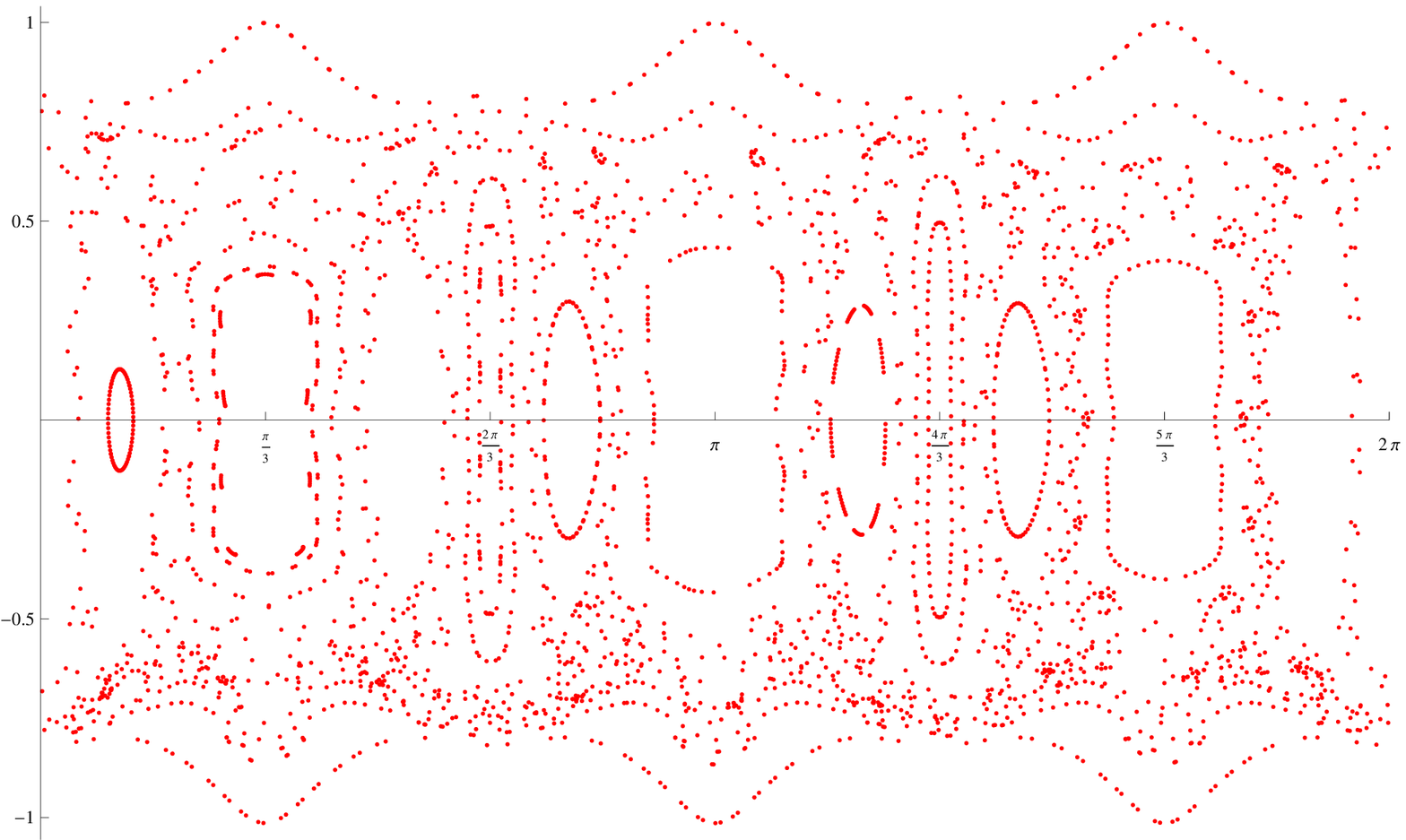} \\
(c)
\includegraphics[width=0.8\textwidth]{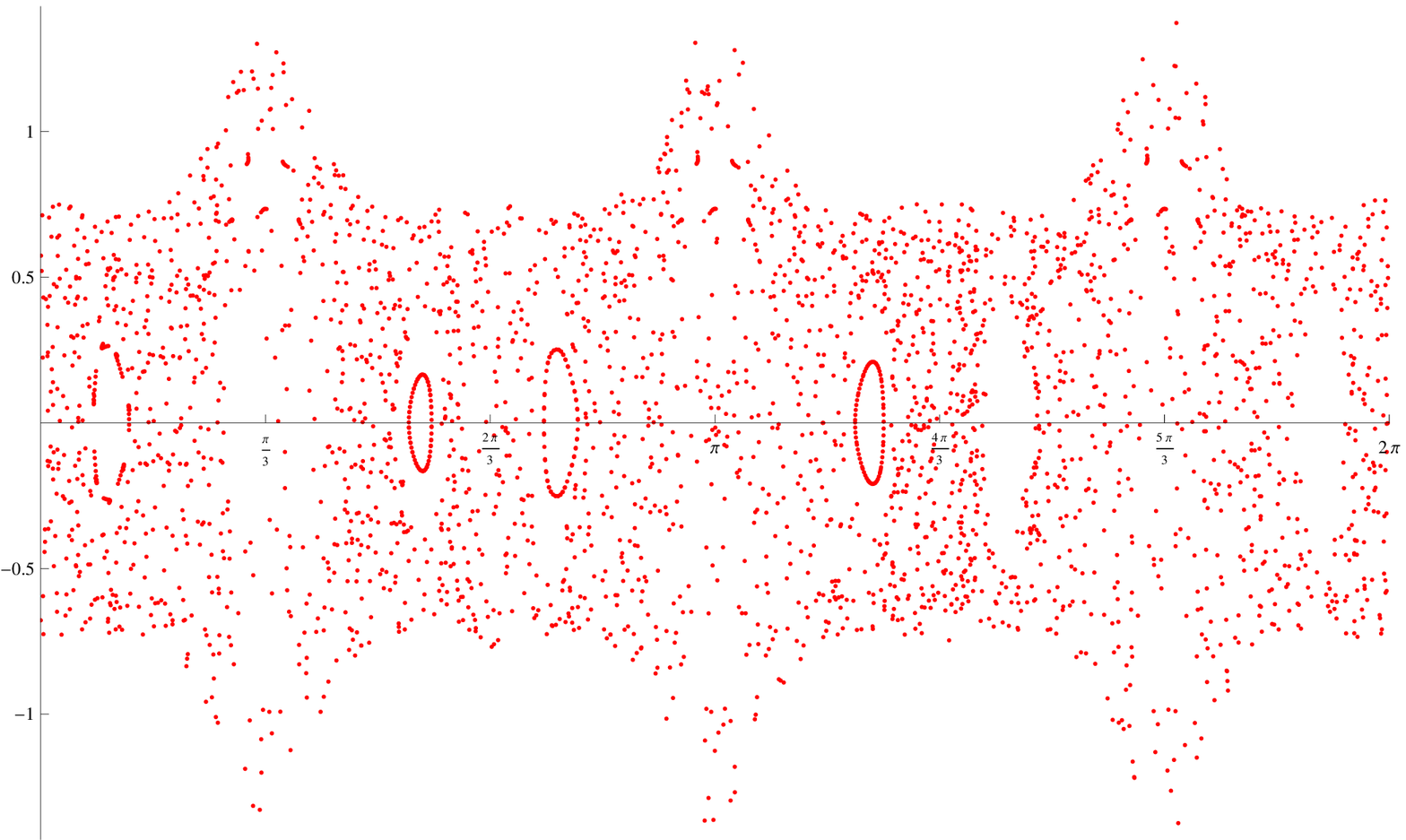}}
\caption{Sample
equatorial Poincar\'{e} sections for the $n=3$ sectoral harmonic
surface, with the following values of $\vep$: (a) $\vep=0.1$; (b)
$\vep=0.2$; (c) $\vep=0.3$. The axes are $\phi$ and $\dot{\phi}$.
See text for discussion.}\label{psecplots}
\end{figure}
\end{center}

The onset of chaotic motion is visible in the Poincar\'{e} section
of all sectoral harmonic surfaces, with the exception of the $n=1$
surface. On this surface the geodesic flow is perfectly regular
for all values of $\vep\in[0,1)$. This is because the $n=1$
sectoral surface is in fact a surface of revolution; it is the
surface formed by revolving the lima\c{c}on around the $x$-axis,
as we may show:

The $n=1$ sectoral surface develops a `dimple' along  the
$x$-axis. Let us rotate the axes (or alternatively the surface)
about the $y$-axis so that the dimple is along the $z$-axis. As
this rotation amounts to $(x,y,z)\to(z,y,-x)$ we can define a new
set of spherical coordinates $(\tb,\pb)$ related to the old
spherical coordinates by \bes \sin\theta=\sqrt{1-\sin^2\tb
\cos^2\pb},\quad \cos\phi=\frac{\cos\tb}{\sqrt{1-\sin^2\tb
\cos^2\pb}}. \ees The $n=1$ sectoral harmonic surface is therefore
given by \be r=1+\vep \sin\theta \cos\phi=1+\vep\cos\tb, \ee the
parametric form of a lima\c{c}on, and clearly $\partial
r/\partial\pb=0$. This surface will therefore have integrable
geodesic flow, a fact that will prove useful in the next section
as a check of the procedure carried out.

The Poincar\'{e} sections provide strong evidence for the
following conclusion: the geodesic equations for the sectoral
harmonic surfaces with $n>1$ are not integrable. However, the
Poincar\'{e} sections cannot be used as proof of this statement,
for the following reasons:\\ 1) as the break up of regular motion
apparently happens only for larger values of $\vep$, there are
possibly small values of $\vep$ for which the equations are
integrable, as in the general Henon-Heiles system for example \cite{mccauley};\\
2) the integrations which generate the plots are run over a finite
time; if the integrations were run over a longer time interval
perhaps a sequence of points which seems to form a closed curve
would begin to wander away. This `stickiness' phenomenon is in
fact quite common in chaotic systems (see for example Murray and
Dermott \cite{murray}), and\\ 3) while great pain may be taken to
ensure accuracy in numerical investigations, there is no escaping
the fact that these are only approximate solutions to a system of
nonlinear ODE's.

Taking these points on board we seek a rigorous and analytical
test of the integrability or otherwise of the geodesic equations.
This will be the goal of the next section.

%% file: kovsec.tex
\section{The variational equations and non-integrability}

While proving a dynamical system is integrable is very
challenging, and amounts to essentially finding the first
integrals, or showing the Hamiltonian is separable in some sense
(see for example \'{O} Math\'{u}na \cite{mathuna}), there are a
number of ways of proving a dynamical system is {\it not}
integrable. This is because integrable systems are very
structured, and must admit a number of properties; showing one of
these properties does not hold suffices to prove the system is not
integrable. While Poincar\'{e} famously showed the
non-integrability of the 3-body problem over 100 years ago
\cite{poincare}, a particularly fruitful modern approach is that
due to Morales and Ramis, as expounded in
\cite{morruiz},\cite{mrthm1} and \cite{mrthm2}, and applied to the
3-body problem in \cite{bouch} and \cite{tsy} among others.

This method focuses on the variational equations which arise from
linearizing around a non-constant solution to the dynamical
system. In loose terms (which we will define more precisely
below), if a Hamiltonian system is integrable then the variational
equations will be solvable, in the sense of differential Galois
theory. Consequently, if the variational equations are not
solvable then the Hamiltonian system is not integrable. This means
we need only focus on the variational equations, which is a great
advantage as the variational equations are of course linear and
therefore much easier to analyse. We will give a brief summary of
the required results first.

\subsection{Solvability and the Morales-Ramis theorem}

Consider the system of linear ordinary  differential equations \be
\xi'=A\xi, \quad '=d/dz \label{sysord} \ee where the elements of
$A$ are in some field $K$ (typically the field of meromorphic
functions, $\mathbb{C}(z)$). We may construct the Picard-Vessiot
extension to this field $L/K$ by adjoining to $K$ the solutions to
\eqref{sysord}. Associated with \eqref{sysord} is the differential
Galois group $\mathcal{G}$ which is the group of automorphisms of
$L/K$ which leave fixed the elements of $K$. The linear system is
{\it solvable} if we may go from the base field $K$ to the field
extension $L/K$ by adjoining to $K$
\begin{enumerate} \item integrals, \item exponentiation of
integrals, and \item algebraic functions
\end{enumerate} of the elements of $K$ (\cite{morruiz} gives a
complete introduction to this topic). Thus we can `build' the
solutions to \eqref{sysord} from the elements of the field that
the coefficients of \eqref{sysord} belong to, by performing a
finite number of `allowed' operations. Solutions of this type may
be referred to as closed form or Louivillian solutions, and the
differential equation may be called solvable or Louivillian.

We define the variational equations as follows: for the dynamical
system  $dx/dt=f(x)$ with non-constant solution $\gamma$, we let
$x=\gamma+\xi$ and the first variational equations are given by
\be \der{\xi}{t}=\left(\pder{f}{x}\right)_\gamma \xi. \ee The
variational equations can be separated into a tangential and a
normal component; as autonomous Hamiltonian systems admit the
Hamiltonian itself as a first integral this is `inherited' by the
tangential variational equation which is therefore always
solvable. And now importantly, the normal variational equation
will inherit any other first integrals, {\it if they exist}, and
will therefore also be solvable. This is the Morales-Ramis
theorem, which we will quote from \cite{morruiz}:

\begin{mrthm} For a $2n$ dimensional Hamiltonian system assume there are $n$ first integrals which are meromorphic,
in involution and independent in the neighborhood of some
non-constant solution. Then the identity component of the
 differential Galois group of the normal variational equation is an abelian subgroup of the symplectic group.
\end{mrthm}

In other words, the normal variational equation (NVE) is solvable
in the sense described above. One advantage of the MR theorem is
that it considers the differential Galois group $\mathcal{G}$
associated with the differential equation \eqref{sysord}, which is
a Lie group, rather than the monodromy group used in Ziglin's
theorem, which is discontinuous.

This theorem suggests a procedure for testing the integrability of
a Hamiltonian system: find a non-constant solution; linearize
around it and separate out the NVE; test the NVE for solvability.
If the NVE is not solvable, then the Hamiltonian system is not
integrable. To test the solvability of the NVE, we will use the
Kovacic algorithm.

\subsection{Kovacic algorithm}

The Kovacic algorithm can be used to find the closed form
solutions of a second order linear ODE with rational coefficients,
and is complete in the sense that if the algorithm does not find
such solutions then none exist. The original source for the
algorithm is Kovacic \cite{kovacicmain}, but see also
\cite{duval}, \cite{mrthm2}, \cite{bardin}, \cite{humanez}, to
name a few, for presentation, discussion, and application of the
algorithm. As we will see, the NVE associated with the system
studied in this work will be Fuchsian, which allows for a more
concise version of the algorithm. As such, we will follow most
closely the presentation of the Kovacic algorithm as given in
Churchill and Rod \cite{church}, and note that the algorithm is
also valid in the non-Fuchsian case.

Consider the linear differential equation \begin{align}
\xi''=r(z)\xi, \quad r\in\mathbb{C}(z). \label{DE} \end{align} If
this equation is Fuchsian we can write \be
r(z)=\sum_{j=1}^k\frac{\beta_j}{(z-a_j)^2}+\sum_{j=1}^k
\frac{\delta_j}{z-a_j}, \label{pfe} \ee where $k$ is the number of
finite regular singular points at locations $z=a_j$. When $\sum
\delta_j=0$ then $z=\infty$ is also a regular singular point, with
$\beta_{\infty}=\sum(\beta_j+\delta_j a_j)$. The indicial
exponents are \[
\tau_j^\pm=\tfrac{1}{2}\left(1\pm\sqrt{1+4\beta_j}\right),\quad
\tau_\infty^\pm=\tfrac{1}{2}\left(1\pm\sqrt{1+4\beta_\infty}\right)
\] at $z=a_j$ and $z=\infty$ respectively. Kovacic
\cite{kovacicmain} proved the following theorem:

\begin{kthm}
Let $\mathcal{G}$ be the differential Galois group associated with
\eqref{DE}, and note $\mathcal{G}\in SL(2,\mathbb{C})$. Then only
one of four cases can hold:

\begin{enumerate}
\item $\mathcal{G}$ is triangulisable (or reducible), in which case \eqref{DE} has
a solution of the form $\xi=e^{\int \omega}$ with
$\omega\in\mathbb{C}(z)$ (we could say $\omega$ is algebraic over
$\mathbb{C}(z)$ of degree 1);
\item $\mathcal{G}$ is conjugate to a subgroup of \[ \left\{
\begin{pmatrix} \lambda & 0 \\ 0 & \lambda^{-1}
\end{pmatrix},\lambda\in\mathbb{C}/\{0\}
\right\}\cup\left\{
\begin{pmatrix} 0 & -\beta^{-1} \\ \beta & 0
\end{pmatrix},\beta\in\mathbb{C}/\{0\}
\right\}, \] in which case \eqref{DE} has a solution of the form
$\xi=e^{\int \omega}$ where $\omega$ is algebraic over
$\mathbb{C}(z)$ of degree 2;
\item $\mathcal{G}$ is finite, in which case \eqref{DE} has a solution of the form
$\xi=e^{\int \omega}$ where $\omega$ is algebraic over
$\mathbb{C}(z)$ of degree 4, 6 or 12;
\item $\mathcal{G}=SL(2,\mathbb{C})$, in which case \eqref{DE} is
not solvable.
\end{enumerate}
\end{kthm}

To show the NVE is not solvable, we must check that each of the
cases 1-3 cannot hold. If this is the case, then
$\mathcal{G}=SL(2,\mathbb{C})$ whose identity component (also
$SL(2,\mathbb{C})$) is not abelian, and therefore the original
Hamiltonian system is not meromorphically integrable according to
the Morales-Ramis theorem.

Kovacic provides an algorithm whereby each case is checked in two
stages: first, we consider certain combinations of the indicial
exponents and retain all those leading to a non-negative integer
$d$, i.e.\ $d\in\mathbb{N}_0$. Second, for each of these values of
$d$ we attempt to construct a monic polynomial of degree $d$, or
possibly a set of polynomials, which satisfy certain equations. If
successful, the solution to \eqref{DE} can be found. We will state
the details for each case in turn:

\medskip

\noindent{\bf Case 1} Define the modified indicial exponents
\begin{align*}
&\alpha_j^\pm=\tau_j^\pm\ \textrm{if}\ \beta_j\neq 0;\quad
\alpha_j^\pm=1\ \textrm{if}\ \beta_j=0\ \textrm{and}\ \delta_j
\neq 0;\quad \alpha_j^\pm=0\ \textrm{if}\ \beta_j=\delta_j=0, \\
&\alpha_\infty^\pm=\tau_\infty^\pm\ \textrm{if}\ \beta_\infty\neq
0; \quad \alpha_\infty^+=1,\alpha_\infty^-=0\ \textrm{if}\
\beta_\infty=0.
\end{align*} Find all combinations of these modified indicial
exponents such that \[
d=\alpha_\infty^\pm-\sum_{j=1}^k\alpha_j^\pm \] is a non-negative
integer. For each $d$ found, look for a monic degree $d$
polynomial $P(z)$ satisfying \[ P''+2\theta
P'+(\theta'+\theta^2-r)P=0, \quad \theta=\sum_{j=1}^k
\frac{\alpha_j^\pm}{z-a_j}. \] If found, a solution to \eqref{DE}
will be $\xi=e^{\int \omega}$ where $\omega=\theta+P'/P$.

\medskip

\noindent{\bf Case 2} Define the following sets: \begin{align*}
&E_j=\{2+e\sqrt{1+4\beta_j},e=0,\pm 2\}\cap\mathbb{Z}\
\textrm{if}\ \beta_j\neq 0; \\
&E_j=\{4\}\ \textrm{if}\ \beta_j=0,\delta_j\neq 0;\quad E_j=\{0\}\
\textrm{if}\ \beta_j=\delta_j=0;\\
&E_\infty=\{2+e\sqrt{1+4\beta_\infty},e=0,\pm 2\}\cap\mathbb{Z}\
\textrm{if}\ \beta_\infty\neq 0; \\
&E_\infty=\{0,2,4\}\ \textrm{if}\ \beta_\infty=0.
\end{align*} Find all combinations of $e_j\in E_j$ and
$e_\infty\in E_\infty$, not all even integers, so that \[
d=\frac{1}{2}\Big(e_\infty-\sum_{j=1}^k e_j \Big) \] is a
non-negative integer. For each $d$ found, look for a monic degree
$d$ polynomial $P(z)$ satisfying \[ P'''+3\theta
P''+(3\theta^2+3\theta'-4r)P'+(\theta''+3\theta\theta'+\theta^3-4r\theta-2r')P=0,\quad
\theta=\frac{1}{2}\sum_{j=1}^k \frac{e_j}{z-a_j}.
\] If found, a solution to \eqref{DE} will be $\xi=e^{\int \omega}$ where
$\omega$ is a root of
$\omega^2+\phi\omega+(\tfrac{1}{2}\phi'+\tfrac{1}{2}\phi^2-r)=0$
and $\phi=\theta+P'/P$.

\medskip

\noindent{\bf Case 3} In this case we must follow the procedure
below for each of the subcases $N=4,6,12$. Define the following
sets:
\begin{align} &F_j=\{6+\tfrac{12e}{N}\sqrt{1+4\beta_j},e=0,\pm
1,\ldots,\pm\tfrac{N}{2}\}\cap\mathbb{Z}\
\textrm{if}\ \beta_j\neq 0; \nonumber \\
&F_j=\{12\}\ \textrm{if}\ \beta_j=0,\delta_j\neq 0;\quad
F_j=\{0\}\
\textrm{if}\ \beta_j=\delta_j=0; \label{fsets}\\
&F_\infty=\{6+\tfrac{12e}{N}\sqrt{1+4\beta_\infty},e=0,\pm
1,\ldots,\pm\tfrac{N}{2}\}\cap\mathbb{Z}. \nonumber
\end{align} Find all combinations of $f_j\in F_j$ and
$f_\infty\in F_\infty$ so that \[
d=\frac{N}{12}\Big(f_\infty-\sum_{j=1}^k f_j \Big) \] is a
non-negative integer. Now define \[
\theta=\frac{N}{12}\sum_{j=1}^k \frac{f_j}{z-a_j}, \quad
S=\prod_{i=1}^k (z-a_j). \] Let $P_N=-P$ where $P(z)$ is a monic
polynomial of degree $d$, and recursively define a set of
polynomials $P_{N-1},P_{N-2},\ldots,P_{-1}$ in the following way:
\[ P_{i-1}=-SP_i'+[(N-i)S'-S\theta]P_i-(N-i)(i+1)P_{i+1},\quad
i=N,N-1,\ldots,0. \] If we find $P_{-1}\equiv 0$, then use the set
of polynomials $P_i$ as coefficients in the following polynomial
and find a solution $\omega$: \[ \sum_{i=0}^N
\frac{S^iP_i}{(n-i)!}\omega^i=0.
\] A solution to \eqref{DE} will be $\xi=e^{\int \omega}$.

\medskip

If these three cases fail to produce a closed form solution, it is
because such solutions do not exist. We are now in a position to
derive the NVE and apply the Kovacic algorithm to it.

\subsection{The normal variational equation}

First we must derive a non-constant  solution to linearize around
(there are no constant solutions to geodesic equations). In
general, it is a difficult task to find an exact solution to
geodesic equations, even for simple symmetrical manifolds. In the
case of the sectoral harmonic surfaces however, we are greatly
aided by the existence of certain planes of symmetry, leading to
planar geodesics. These planes of symmetry play the same role in
this context as invariant planes play in more general contexts. In
what follows we will take the equatorial planar geodesic as our
nominal solution (see Appendix A for some comments on the use of
meridianal planar geodesics).

In the equatorial plane $\theta=\tfrac{\pi}{2}, \dot{\theta}=0$,
the geodesic equation is (from \eqref{geoth},\eqref{geoph} and
noting that $\wt{\Gamma}^\theta_{\phi\phi}=0$, where we use a
tilde to denote the value on $\theta=\tfrac{\pi}{2}$) \be
\ddot{\phi}+\wt{\Gamma}^\phi_{\phi\phi}\ \dot{\phi}^2=0 \ee which
admits Hamiltonian \be 2\mathcal{H}=
\wt{g}_{\phi\phi}\dot{\phi}^2=1. \label{redham} \ee To solve for
this equatorial planar geodesic we would need to rewrite
\eqref{redham} as \be \int ds=\int d\phi \sqrt{\wt{g}_{\phi\phi}}
=\int d\phi \sqrt{(1+\vep\cos(n\phi))^2+\vep^2 n^2 \sin^2(n\phi)}
\label{elliptic}\ee and solve to find $s=s(\phi)$, then invert
this relationship to find $\phi=\phi(s)$. The integrals in
\eqref{elliptic} can be written in terms of elliptic integrals of
various kinds, at least for some values of $n$, however the
analysis would benefit from a unified treatment with $n$ as a
parameter. That the integrals in \eqref{elliptic} are daunting
should come as no surprise: we are essentially trying to find the
arc-length parameterization of the curve $r=1+\vep\cos(n\phi)$,
which we could describe as a lima\c{c}on with $n$ `lobes', and
there are few curves for which the arc-length is known exactly.

Fortunately we do not need to solve \eqref{elliptic} exactly;  for
the sake of the NVE we can make a change of independent variable
which requires then only \eqref{redham}, as we will show. Let us
denote the planar geodesic by \be
(\theta,\phi,\dot{\theta},\dot{\phi})=(\tfrac{\pi}{2},\phi_0,0,\dot{\phi}_0),
\ee and let $\theta=\tfrac{\pi}{2}+\xi,\ \phi=\phi_0+\eta$.
Linearizing around this solution we find the equation in
$\ddot{\xi}$ naturally decouples from that of $\ddot{\eta}$, and
since the equatorial/invariant plane is given by
$\theta=\tfrac{\pi}{2}$ the variation normal to this plane is
given by $\xi$, and thus the NVE is \be
\ddot{\xi}=&\left(-\wt{\Gamma}^{\theta}_{\phi\phi,\theta}\dot{\phi}^2\right)\,\xi+
\left(-2\wt{\Gamma}^{\theta}_{\theta\phi}\dot{\phi}\right)\,\dot{\xi},
\label{nve1} \ee where the coefficients are evaluated along the
nominal equatorial geodesic. Now we make the change of independent
variable $z=\vep\cos(n \phi(s))$ so that \be \der{}{s}=-\vep
n\sin(n\phi)\der{\phi}{s}\der{}{z}=-\frac{\vep n
\sin(n\phi)}{\sqrt{\wt{g}_{\phi\phi}}}\der{}{z} \ee and so on. We
may replace instances of $\cos(n\phi)$ with $z/\vep$, to give an
NVE with rational polynomial coefficients (this process is
sometimes referred to as algebrization, and preserves the identity
component of the differential Galois group \cite{simon}), which we
may write as (letting $\xi'=d\xi/dz$ etc.) \be
\xi''+p(z)\xi'+q(z)\xi=0. \label{nve2} \ee Finally, we transform
this equation into the standard form required for the Kovacic
algorithm \[ \xi''=r(z)\xi
\] where \[ r=-q+\tfrac{1}{4}p^2+\tfrac{1}{2}p' \] and we extend $z$ to the complex domain. We are now in
a position to state the main result of his paper:

\begin{thm}
The geodesic equations for sectoral harmonic surfaces with $n>1$
are not integrable.
\end{thm}
\begin{proof}
When $n>1$, the NVE has 6 regular singular points: five finite
regular singular points and one at infinity. The finite regular
singular points are \be
\{a_1,\ldots,a_5\}=\{-1,\vep,-\vep,\rho_+,\rho_-\} \label{app1}
\ee with \be \rho_{+/-}=\frac{1\pm
n\sqrt{1+\vep^2(n^2-1)}}{n^2-1}. \label{app2}\ee The case $n=1$ is
special, and shall be dealt with below. When $n>1$ and $0<\vep<1$
all singular points are real and distinct. The coefficients in the
partial fraction expansion \eqref{pfe} are \be
\beta_i=\left\{0,\frac{-3}{16},\frac{-3}{16},\frac{5}{16},\frac{5}{16}
\right\}, \quad \beta_\infty=\frac{n+1}{n^2}, \ee and the
$\delta_i$'s are given in Appendix B; here we need only mention
that $\sum \delta_i=0$.

We must now check all three cases of the Kovacic algorithm to see
is the NVE solvable. Often in the literature some of the cases can
be ruled out straight away, based on the ``necessary conditions''
given in Section 2.1 of Kovacic \cite{kovacicmain}. However all of
these necessary conditions are satisfied in the present case since
the order of each pole of $r(z)$ is 2 and moreover \be
\sqrt{1+4\beta_i}=\Big\{1,\frac{1}{2},\frac{1}{2},\frac{3}{2},\frac{3}{2}
\Big\}\in\mathbb{Q},\quad\sqrt{1+4\beta_\infty}=\frac{n+2}{n}\in\mathbb{Q};
\ee therefore we must examine each case.

The first stage in each case consists of finding all combinations
of the $\beta$ coefficients such that $d$ is a non-negative
integer. Unfortunately, there are a {\it lot} of such
combinations. We show the combinations in Table 1, where $N=1$ and
$N=2$ represent cases 1 and 2 respectively, and $N=4,6,12$
represent case 3. The values of $n$, the specific sectoral
harmonic under consideration, are shown in the column on the left.
Each entry in the table gives the possible values of $d$ in bold,
and in brackets the number of combinations of indicial exponents
which lead to this value of $d$. As can be seen immediately from
the table, when $n=7,8,9,11$ and $n>12$ there are no combinations
leading to $d\in\mathbb{N}_0$; as such the NVE in these cases is
not solvable.

For the other values of $n$ we must try to construct a monic
polynomial of degree $d$ which solves the required equations.
Clearly there are too many to go through individually; for the
sake of demonstration we will describe the most involved case,
$n=2, N=12$. The relevant sets defined in \eqref{fsets} are \be
&F_1=\{12\},\quad F_2=F_3=\{6,7,5,8,4,9,3\},\quad
F_4=F_5=\{6,9,3,12,0,15,-3\}, \nonumber \\
&F_\infty =\{6,8,4,10,2,12,0,14,-2,16,-4,18,-6\} \ee  It is easy
to see how there can be many combinations leading to
$d\in\mathbb{N}_0$. The largest value comes from the combination
\[ d=18-(12+3+3-3-3)=6. \]
In this case, we let $P=z^6+p_5z^5+\ldots+p_0$, and define
\begin{align*}
P_{12}&=-P, \\
P_{11}&=-S P_{12}'-S\theta P_{12},\\
P_{10}&=-S P_{11}'+[S'-S\theta]P_{11}-12 S^2 r P_{12},\\
P_{9}&=\ldots \end{align*} all the way down to $P_{-1}$. These
computations quickly become very large and cumbersome. A useful
note is that each $P_i$ is simply a polynomial in $z$, albeit one
with very complicated coefficients. The order increases in a
regular fashion: since $S=O(z^k)$, $S\theta=O(z^{k-1})$ and $S^2
r=O(z^{2(k-1)})$ where $k$ is the number of finite regular
singular points, we see that $P_{N}=O(z^d), P_{N-1}=O(z^{k-1+d})$
and so on until $P_{-1}=O(z^{(N+1)(k-1)+d})$. In the present case
this means $P_{-1}=O(z^{58})$, although the leading coefficient
always vanishes. Another useful point is that we only need the
first few coefficients in $P_{-1}$; setting the first equal to
zero gives a value for $p_5$, the second gives a value for $p_4$
and so on until we reach a term which cannot be set equal to zero.
Then we know that $P_{-1}\not\equiv 0$, which means that this
combination for $d$ does not produce a solution.

This procedure has been followed for all combinations listed in
Table 1 (using Mathematica). For each and every case, it is not
possible to construct the required polynomial $P$. As such, the
NVE arising from linearising about an equatorial geodesic on the
sectoral harmonic surfaces with $n>1$ does not have a differential
Galois group belonging to cases 1,2 or 3 in Kovacic's theorem;
instead, $\mathcal{G}=SL(2,\mathbb{C})$ whose identity component
is not abelian. Therefore according to the Morales-Ramis theorem
the geodesic equations are not integrable. This ends the proof.
\end{proof}

\begin{center}
\begin{table}
\begin{tabular}{|c||l|l|l|l|l|}
\hline $n$ & $N=1$ & $N=2$ & $N=4$ & $N=6$ & $N=12$ \\
\hline \multirow{2}{*}{2} & \multirow{2}{*}{{\bf 0}(4)} &
\multirow{2}{*}{{\bf 0}(3), {\bf 1}(1)} & {\bf 0}(4), {\bf 1}(2),
& {\bf 0}(21), {\bf 1}(10), & {\bf 0}(31),
 {\bf 1}(20), {\bf 2}(13), \\
 & & & {\bf 2}(1) & {\bf 2}(3), {\bf 3}(1) & {\bf 3}(8), {\bf 4}(4), {\bf 5}(2), {\bf 6}(1) \\ \hline
\multirow{2}{*}{3} & \multirow{2}{*}{-} & \multirow{2}{*}{-} &
\multirow{2}{*}{-} & \multirow{2}{*}{{\bf 0}(9), {\bf 1}(3), {\bf
2}(1)}
& {\bf 0}(9), {\bf 1}(6), {\bf 2}(3), \\
 & & & & & {\bf 3}(2), {\bf 4}(1) \\ \hline
\multirow{2}{*}{4} & \multirow{2}{*}{-} & \multirow{2}{*}{{\bf
0}(2)} & \multirow{2}{*}{{\bf 0}(2), {\bf 1}(1)} &
\multirow{2}{*}{{\bf 0}(7),
 {\bf 1}(2)} & {\bf 0}(7), {\bf 1}(3), {\bf 2}(2), \\
& & & & & {\bf 3}(1) \\ \hline 5 & - & - & - & - & {\bf 0}(2), {\bf 1}(1) \\
\hline 6 & - & - & - & {\bf 0}(3), {\bf 1}(1) & {\bf 0}(3), {\bf 1}(2), {\bf 2}(1) \\
\hline
7 & - & - & - & - & - \\
8 & - & - & - & - & - \\
9 & - & - & - & - & - \\ \hline 10 & - & - & - & - & {\bf 0}(1) \\
\hline 11 & - & - & - & - & - \\ \hline
12 & - & - & - & {\bf 0}(2) & {\bf 0}(2), {\bf 1}(1) \\
\hline
\end{tabular}
\caption{The possible values of the non-negative integer $d$; see
text for an explanation.}
\end{table}
\end{center}
\vspace{-1cm}

We can now deal with the special case $n=1$:

\begin{thm}
The equatorial NVE on the $n=1$ sectoral harmonic surface is
solvable.
\end{thm}
\begin{proof}
When $n=1$ the equatorial NVE has five regular singular points,
four finite at $a_j=\{-1,\vep,-\vep,\rho\}$ with
$\rho=-\tfrac{1}{2}(1+\vep^2)$, and one at infinity. We calculate
\be \beta_j=\Big\{0,\frac{-3}{16},\frac{-3}{16},\frac{5}{16}
\Big\},\quad \beta_\infty=\frac{45}{16} \ee and hence \be
\alpha_j^\pm=\Big\{
(1,1),\big(\tfrac{3}{4},\tfrac{1}{4}\big),\big(\tfrac{3}{4},\tfrac{1}{4}\big),\big(\tfrac{5}{4},-\tfrac{1}{4}\big),
\Big\},\quad
\alpha_\infty^\pm=\big(\tfrac{9}{4},-\tfrac{5}{4}\big). \ee There
is a combination giving \be
d=\frac{9}{4}-\Big(1+\frac{3}{4}+\frac{3}{4}-\frac{1}{4}\Big)=0
\ee and since $\theta$ thus defined solves the Riccatti equation
$\theta'+\theta^2-r=0$ we have the pair of Louivillian solutions
\be \xi_1=C_1(z+1)(z^2-\vep^2)^{3/4}(z-\rho)^{-1/4}, \quad
\xi_2=C_2 \xi_1 \int \xi_1^{-2}dz. \ee This is of course entirely
as expected since the $n=1$ sectoral surface is a surface of
revolution and thus has integrable geodesic flow.

\end{proof}